\documentclass[11pt]{article}
\usepackage{latexsym,amsfonts,amssymb,amsmath,amsthm,verbatim}
\usepackage{graphicx}

 \newtheorem{theorem}{Theorem}[section] \newtheorem{%
definition}[theorem]{Definition}
\newtheorem{lemma}[theorem]{Lemma}
\newtheorem{corollary}[theorem]{Corollary} \newtheorem{remark}[theorem]{%
Remark}  %
\numberwithin{equation}{section}

\newcommand\RR{\ensuremath{\mathbb{R}}}

\newcommand\ZZ{\ensuremath{\mathbb{Z}}}

\newcommand\PP{\ensuremath{\mathbb{P}}}

\begin{document}
\begin{center}

\textbf{\Large Asymptotic  Dynamics of a Class of Coupled  Oscillators}

 \vskip 0.2cm

\textbf{\Large  Driven by White Noises}

\vskip 1cm

{\large Wenxian Shen$^{a}$ \footnote{The first author is partially
supported by NSF grant DMS-0907752}}, Zhongwei Shen$^{a}$, Shengfan
Zhou$^{b}$ \footnote{The third author is supported by National
Natural Science Foundation of China under Grant 10771139, and the
Innovation Program of Shanghai Municipal Education Commission under
Grant 08ZZ70}

\vskip 0.3cm

$^{a}$\textit{Department of Mathematics and Statistics, Auburn
University},

\textit{Auburn 36849, USA}

$^{b}$\textit{Department of Applied Mathematics, Shanghai Normal
University},

\textit{Shanghai 200234, PR China}

\vskip 1cm

\begin{minipage}[c]{13cm}

\noindent \textbf{Abstract}: This paper is devoted to the study of
the asymptotic dynamics of a class of coupled  second order
oscillators driven by white noises. It is shown that  any system of
such coupled oscillators with positive damping and coupling
coefficients  possesses a global random attractor. Moreover, when
the damping and the coupling coefficients are sufficiently large,
the global random attractor is a one-dimensional random horizontal
curve regardless of the strength of the noises, and the system has a
rotation number, which implies that the oscillators in the system
tend to oscillate with the same frequency eventually and therefore
the so called frequency locking is successful. The results obtained
in this paper generalize many existing results on the asymptotic
dynamics for a single second order noisy oscillator to systems of
coupled second order noisy oscillators. They show that coupled
damped second order oscillators with large damping have similar
asymptotic dynamics as the limiting coupled first order oscillators
as the damping goes to infinite and also that coupled damped second
order oscillators have similar asymptotic dynamics as their  proper
space continuous counterparts, which are of great practical
importance.

\vspace{5pt}

\noindent\textbf{Keywords}: Coupled second order oscillators; white
noises; random attractor; random horizontal curve;
rotation number; frequency locking

\vspace{5pt}

\noindent\textbf{AMS Subject Classification}: 60H10, 34F05, 37H10.

\end{minipage}
\end{center}

\section{Introduction}

 This paper is devoted to the study of
the asymptotic dynamics of the following system of second order oscillators driven by additive
noises:
\begin{equation}\label{main-eq}
d\dot{u}_j+\alpha du_j+K(Au)_{j}dt+\beta
g(u_j)dt=f_jdt+\epsilon_jdW_{j},
\end{equation}
where  $j\in \mathbb{Z}^{d}_N:=\{j=(j_1,\dots,j_d)\in\mathbb{Z}^{d}:1\leq
j_1,\dots,j_d\leq N\}$,  $u_j$ is a scalar unknown function of $t$ and $u=(u_1,u_2,\cdots,u_{N^d})^\top$, $\alpha$ and $K$ are positive constants, $A$ is an $N^d\times N^d$ matrix and
$(Au)_j$ stands for the $j$th
component of the vector $Au$,  $\beta\in\mathbb{R}$, $g$ is  a periodic function,
$f_j$ and $\epsilon_j$ are constants,
and $\{W_j(t)\}_{j\in\mathbb{Z}^{d}_{N}}$ are
independent two-sided real-valued Wiener processes. Moreover,  $A$ and $g$ satisfy

\medskip
\noindent{\bf(HA)} {\it
 $A$ is an
$N^d\times N^d$ nonnegative definite symmetric matrix with
eigenvalues denoted by $\lambda_i$, $i=0,1,\dots,N^d-1$ satisfying that
\begin{equation*}
0=\lambda_0<\lambda_1\leq\cdots\leq\lambda_{N^d-1},
\end{equation*}
$\lambda_0$ is algebraically simple,
and $(1,\dots,1)^{\top}\in\mathbb{R}^{N^d}$ is an
eigenvector corresponding to $\lambda_0=0$.
}

\medskip

\noindent{\bf (HG)} {\it $g\in
C^1(\mathbb{R},\mathbb{R})$ has the following properties
\begin{equation*}
g(x+\kappa)=g(x),\quad|g(x)|\leq c_1,\quad|g'(x)|\leq
c_2,\,\,\forall x\in\mathbb{R},
\end{equation*}
where $c_1>0$, $c_2>0$ and $\kappa>0$ is the smallest positive
period of $g$.}

System \eqref{main-eq} appears in many applied problems including Josephson junction arrays
and coupled pendula (see \cite{HaBeWi}, \cite{Lev}, \cite{QZQ},
\cite{WiHa}, etc.).
Physically, $\alpha$ in \eqref{main-eq} represents the damping of
the system and $K$ is the coupling coefficient of the system. \eqref{main-eq} then represents a  system of
$N^d$ coupled damped oscillators independently driven by white
noises.

System \eqref{main-eq} also arises from various spatial
discretizations of certain damped hyperbolic partial differential
equations. For example, the $N^d\times N^d$ matrix $A$ in
\eqref{main-eq} includes the discretization of negative Laplace
operator $-\Delta$ with Neumann or periodic boundary conditions
defined as follows:
\begin{equation*}
\begin{split}
(Au)_j=(Au)_{(j_1,j_2,\ldots,j_d)}&=\frac{1}{h^2}\big[2du_j-u_{(j_1+1,j_2,\ldots,j_d)}-u_{(j_1,j_2+1,\ldots,j_d)}-\cdots
-u_{(j_1,j_2,\ldots,j_d+1)}\\
&\quad-u_{(j_1-1,j_2,\ldots,j_d)}-u_{(j_1,j_2-1,\ldots,j_d)}-\cdots-u_{(j_1,j_2,\ldots,j_d-1)}\big],
\end{split}
\end{equation*}
with Neumann boundary condition
\begin{equation*}
\begin{split}
u_{(j_1,\dots,j_{i-1},0,j_{i+1},\dots,j_{d})}&=u_{(j_1,\dots,j_{i-1},1,j_{i+1},\dots,j_{d})},\\
u_{(j_1,\dots,j_{i-1},N+1,j_{i+1},\dots,j_{d})}&=u_{(j_1,\dots,j_{i-1},N,j_{i+1},\dots,j_{d})}
\end{split}
\end{equation*}
or periodic boundary condition
\begin{equation*}
\begin{split}
u_{(j_1,\dots,j_{i-1},0,j_{i+1},\dots,j_{d})}&=u_{(j_1,\dots,j_{i-1},N,j_{i+1},\dots,j_{d})},\\
u_{(j_1,\dots,j_{i-1},N+1,j_{i+1},\dots,j_{d})}&=u_{(j_1,\dots,j_{i-1},1,j_{i+1},\dots,j_{d})}
\end{split}
\end{equation*}
for $j=(j_1,\dots,j_d)\in\mathbb{Z}^{d}_{N}$ and $i=1,\dots,d$.
Thus, \eqref{main-eq} with  $u_j=u(j_1 h,\cdots,j_i h,\cdots,j_d h)$ ($h=L/N$),  $A$ being as above,
$f_j=f$,  $\epsilon_j=\epsilon$, and $W_j=W$
is a spatial discretization of the following problem
\begin{equation}\label{sine-gordon-eq}
d\dot{u}+\alpha d{u}-K\Delta u dt+\beta g(u) dt=f dt+\epsilon dW,\quad\text{in}\,\,
U\times\mathbb{R}^{+}
\end{equation}
with Neumann boundary condition or periodic boundary condition,
i.e.,
\begin{equation*}
\frac{\partial u}{\partial n}=0\quad\text{on}\,\,\partial
U\times\mathbb{R}^{+}
\end{equation*}
or
\begin{equation*}
u|_{\Gamma_j}=u|_{\Gamma_{j+d}},\quad \frac{\partial u}{\partial
x_j}\Big|_{\Gamma_j}=\frac{\partial u}{\partial
x_j}\Big|_{\Gamma_{j+d}},\quad j=1,\dots,d,
\end{equation*}
where $\Gamma_j=\partial U\cap\{x_j=0\}$, $\Gamma_{j+d}=\partial
U\cap\{x_j=L\}$, $j=1,\dots,d$ and $U=\prod_{i=1}^{d}(0,L)$. Note
that if $g(u)=\sin u$, \eqref{sine-gordon-eq} is the so called
damped sine-Gordon equation, which is used to model, for instance,
the dynamics of a continuous family of junctions (see \cite{Tem}).

Two  of the main dynamical aspects about coupled oscillators and damped wave equations considered
in the literature are the existence and structure of global attractors and the
phenomenon of frequency locking.
A large amount of research has been carried out toward these two aspects for
a variety of systems related to \eqref{main-eq}. See for example,
\cite{QSZ1,QSZ2, QQZ, QZQ, Z} for the study  of coupled oscillators
with constant or periodic external forces;
\cite{Hal, QQWZ, QZZ, Tem,  WaZh1} for the study  of the deterministic damped sine-Gordon equation;
\cite{CSZ,  QW, S} for the study of coupled oscillators driven by white noises;
 and \cite{Fan, SZS, ZhYiOu} for the study of stochastic damped sine-Gordon equation.
 Many of the existing works focus on the existence of global attractors and the estimate
 of the dimension of the global attractors. In \cite{CSZ, S, SZS},  the existence and structure of random attractors
 of stochastic oscillators and stochastic damped wave equations
 are studied.
 In particular,
 the asymptotic dynamics of a single second order noisy
oscillator, i.e., \eqref{main-eq} with $N=1$, is studied in \cite{S}.
The author of \cite{S} proved the existence of a random attractor which is a family of horizontal
curves and the existence of a rotation number which implies the
frequency locking.
 In \cite{CSZ}, the authors considered a class of
coupled first order oscillators driven by white noises.  Among
those, the existence of a one-dimensional random attractor and the
existence of a rotation number are  proved in \cite{CSZ}.
The system of coupled first order oscillators considered in \cite{CSZ}
is of the form
\begin{equation}
\label{first-order}
du_j+K(Au)_jdt+\beta g(u_j)dt=f_jdt+\epsilon d W_j,\quad j\in \ZZ_N^d.
\end{equation}
Note that,  by resealing the time variable by $t\to \frac{t}{\alpha}$,
\eqref{main-eq} becomes
\begin{equation}
\frac{1}{\alpha}d\dot{u}_j+du_j+K(Au)_jdt+\beta
g(u_j)dt=f_jdt+\epsilon d W_j,\quad j\in \ZZ_N^d.
\end{equation}
Hence, \eqref{first-order}
can be formally viewed as the limiting system of \eqref{main-eq} as the damping coefficient $\alpha$
goes to infinite. In \cite{SZS}, the authors investigated the
existence and structure of random attractors  of damped sine-Gordon equations of the form \eqref{sine-gordon-eq}
with Neumann boundary condition, which is a space continuous counterpart of \eqref{main-eq}
as mentioned above.

However, many important dynamical aspects including the existence of global attractor  and
the occurrence of frequency locking have been hardly studied for coupled second order oscillators of the form \eqref{main-eq}
driven by
white noises.
It is of great interest to investigate the extent to which the
existing  results on asymptotic dynamics of a single  second order noisy
oscillator may be generalized to systems of coupled  second order noisy oscillators.
Thanks to the relations  between \eqref{main-eq} and \eqref{first-order} and between \eqref{main-eq} and \eqref{sine-gordon-eq},
 it is also of great interest to explore
the similarity and difference between the dynamics of coupled  damped  second order oscillators and  its limiting coupled first order
oscillators as the damping coefficient goes to infinite and between the dynamics of coupled damped  second order oscillators and
 their proper space continuous counterparts.
The objective of this paper is to
 carry out a study along this line. In particular,  we study the asymptotic or global dynamics of \eqref{main-eq},
including  the existence and structure of global attractor in proper phase space and the success of frequency locking.

In order to do so,
as usual, we first change \eqref{main-eq} to some  system of coupled first order random equations.
Assume $N\geq2$ and $d\geq1$ ($N=1$ reduces to the single noisy oscillator case considered in \cite{S}).
Let $u=(u_j)_{j\in\mathbb{Z}_N^d}$,
$g(u)=(g(u_j))_{j\in\mathbb{Z}_N^d}$,
$f=(f_j)_{j\in\mathbb{Z}_N^d}$,
$W(t)=(\epsilon_jW_j(t))_{j\in\mathbb{Z}_N^d}$. Then,
\eqref{main-eq} can be written as the following matrix form,
\begin{equation}\label{SDE}d\dot{u}+\alpha
du+KAudt+\beta g(u)dt=fdt+dW(t).
\end{equation}
Let
\begin{equation*}
\Omega_j=\Omega_0 =\{\omega_0\in
{C}(\mathbb{R},\mathbb{R}):\omega_0(0)=0\}
\end{equation*}
equipped with the compact open topology,
$\mathcal{F}_j=\mathcal{B}(\Omega_0)$ be the Borel $\sigma$-algebra
of $\Omega_0$ and $\PP_j$ be the corresponding Wiener measure for
$j\in\mathbb{Z}^{d}_{N}$. Let
$\Omega=\prod_{j\in\mathbb{Z}^{d}_{N}}\Omega_j$, $\mathcal{F}$ be
the product $\sigma $-algebra on $\Omega$ and $\PP$ be the induced
product Wiener measure. Define $(\theta _t)_{t\in\mathbb{R}}$ on
$\Omega$ via
\begin{equation*}
\theta _t\omega (\cdot )=\omega (\cdot +t)-\omega (t),\quad
t\in\mathbb{R}.
\end{equation*}
Then, $(\Omega,\mathcal{F},\PP,(\theta_t)_{t\in\mathbb{R}})$ is an
ergodic metric dynamical system (see \cite{A}).
Consider the Ornstein-Uhlenbeck equation,
\begin{equation}
\label{OU-eq}
dz+zdt=dW(t),\quad z\in\RR^{N^d}.
\end{equation}
Let $z(\theta_t\omega)=(z_j(\theta_t\omega))_{j\in\mathbb{Z}_N^d}$
be the unique stationary solution of \eqref{OU-eq} (see \cite{A,BLW, DLS}
for the existence and various properties of $z(\cdot)$). Let
$v=\dot{u}-z(\theta_t\omega)$. We obtain the following  equivalent system of
\eqref{SDE},
\begin{equation} \label{RDE}
\left\{ \begin{aligned}
&\dot{u}=v+z(\theta_t\omega),\\
&\dot{v}=-KAu-\alpha v+f-\beta g(u)+(1-\alpha)z(\theta_t\omega).
\end{aligned} \right.
\end{equation}

To study the global dynamics of \eqref{main-eq}, it is therefore equivalent to
study the global dynamics of \eqref{RDE}. Observe that the natural phase space
for \eqref{RDE} is $E:=\RR^{N^d}\times\RR^{N^d}$ with the standard
Euclidean norm.  Thanks to the presence of the damping,
it is expected that \eqref{RDE} possesses a global attractor in certain sense. However,
due to
the uncontrolled component of the solutions along the direction of
the eigenvectors of the linear operator in the right of \eqref{RDE} corresponding to the zero eigenvalue,
 there is no bounded attracting sets in $E$ with
 the standard
Euclidean norm, which will lead to nontrivial dynamics. There is
also some additional difficulty if one  studies  \eqref{RDE} in $E$
with the standard Euclidean norm due to the zero limit of some
eigenvalues of the linear operator in the right of \eqref{RDE} as
$\alpha\to\infty$. The later difficulty does not appear for coupled
first order oscillators studied in \cite{CSZ} and for a single noisy
oscillator considered in \cite{S}. We will overcome the difficulty
by using some equivalent norm on $E$ and considering \eqref{RDE} in
some proper quotient space of $E$  and prove the existence of a
global random attractor as well as the existence of a rotation
number of \eqref{RDE}.

To be more precise,  let
\begin{equation}
\label{c-eq}
C=\begin{pmatrix}0&I\\-KA&-\alpha
I\end{pmatrix}.
\end{equation}
 By simple matrix analysis,
the eigenvalues of $C$ are given by (see \cite{LZ,QQZ} for
example)
\begin{equation}\label{eigenvalueC}
\mu_{i}^{\pm}=\frac{-\alpha\pm\sqrt{\alpha^2-4K\lambda_i}}{2},\quad
i=0,1,\dots,N^d-1.
\end{equation}
Note that $\mu_0^+=0$, which requires some special consideration for
the solutions along the direction of the eigenvector
$\eta_0=(1,\dots,1,0,\dots,0)^{\top}$ corresponding to $\mu_0^+$. We
 overcome this difficulty by considering \eqref{RDE} in the
cylindrical space $E_{1}/{\kappa \eta_0\mathbb{Z}}\times E_{2}$,
where $E_{1}=\text{span}\{\eta_0 \}$, $E_{2}$ is the space spanned by all
the eigenvectors  corresponding to non-zero eigenvalues of $C$ (see
section 4 for details). We then  prove

\medskip
\noindent (1) {\it  For any  $\alpha>0$ and $K>0$,  system
\eqref{main-eq} possesses a global random attractor (which is unbounded along the one-dimensional space $E_1$ and bounded along
the one-codimensional space
 $E_2$) (see Theorem
\ref{existence-random-attractor2}, Corollary
\ref{existence-random-attractor-corollary} and Remark 4.4). }

\medskip

 It is expected physically that when the damping coefficient $\alpha\to \infty$, the dynamics of \eqref{RDE}
becomes simpler or the structure of the global attractor of
\eqref{RDE} becomes simpler. However, $\mu_i^+\to 0$ as $\alpha\to
\infty$ for $i=1,2,\cdots,N^d-1$, which gives  rise to some
difficulty for studying the structure of the global attractor in $E$
with the standard Euclidean norm. We  introduce an equivalent norm
on $E$ to overcome this difficulty (see section 3 for the
introduction of the equivalent norm, the choice of such equivalent norm was first discovered in \cite{Mor}) and prove

\medskip
\noindent (2) {\it When $\alpha$ and $K$ are sufficiently large, the global  random attractor of
\eqref{main-eq} is a one-dimensional random horizontal curve (see
Theorem \ref{one-dimension-thm} and Corollary
\ref{one-dimension-cor}), and the rotation number (see Definition
\ref{definition-rotation-number}) of \eqref{main-eq} exists (see
Theorem \ref{existence-rotation-number1} and Corollary
\ref{existence-rotation-number1-corollary}).
}

\medskip

Note that roughly  a real number $\rho\in\RR$ is called the {\it rotation number} of \eqref{main-eq} or \eqref{RDE} if
 for any solution $\{u_j(t)\}_{j\in \ZZ_N^d}$
of \eqref{main-eq}, the limit
$\lim_{t\to\infty}\frac{u_j(t)}{t}$ exists almost surely for any $j\in\ZZ_N^d$  and
\begin{equation*}
\lim_{t\to\infty}\frac{u_j(t)}{t}=\rho \quad {\rm for}\quad a.e. \quad \omega\in\Omega\quad {\rm and}\quad j=1,2,\cdots,N^d
\end{equation*}
(see  Definition \ref{definition-rotation-number} and the remark after Definition \ref{definition-rotation-number}). Hence if \eqref{main-eq} has a rotation number, then
 the oscillators in the system tend to oscillate with the
same frequency eventually and therefore the so called frequency locking is
successful.

(1) and (2) above are the main results of the paper. They make an
important contribution to the understanding of  coupled
second order oscillators  driven by noises. Property (1) shows that system
\eqref{main-eq} is dissipative along the one-codimensional space $E_2$. By property (2), the asymptotic
dynamics of \eqref{main-eq} with sufficiently large $\alpha$ and $K$
is one dimensional regardless of the strength of noise. Property (2) also shows that all the solutions of
\eqref{main-eq} tend to oscillate with the same frequency eventually
almost surely and hence frequency locking is successful in
\eqref{main-eq} provided that $\alpha$ and $K$ are sufficiently
large.

The results obtained in this paper  generalize  many existing results on the asymptotic dynamics for a single damped noisy oscillator to
systems of coupled damped noisy oscillators. They show that coupled damped second order oscillators with large damping
have similar asymptotic dynamics as the limiting coupled first order oscillators as the damping goes to infinite
and hence one may use coupled first order oscillators to analyze qualitative properties of
coupled second order oscillators with large damping, which is of great practical importance.
They also show that coupled damped second order oscillators have similar asymptotic dynamics as their  proper space
continuous counterparts and hence one may use finitely many coupled oscillators to study qualitative properties of
 damped wave equations, which is of great practical importance too.

The rest of the paper is organized as follows. In section 2, we
present  some basic concepts and properties for general random
dynamical systems. In section 3, we provide some basic settings
about \eqref{main-eq} and show that it generates a random dynamical
system. We prove in section 4 the existence of a global random attractor
of the random dynamical system $\phi$ generated by \eqref{main-eq}
for any $\alpha>0$ and $K>0$.  We show in section 5 that the global random
attractor of $\phi$ is a random horizontal curve
and show in section 6 that \eqref{main-eq} has  a rotation number, respectively,
provided that
$\alpha$ and $K$ are sufficiently large.

\section{Random Dynamical Systems}
In this section, we collect some basic knowledge
about general random dynamical system (see \cite{A,C} for details).
Let $(X,d)$ be a complete and separable metric space with Borel
$\sigma$-algebra $\mathcal{B}(X)$.
\begin{definition}\label{random-dynamical-system-def}
A {\rm continuous random dynamical system over
$(\Omega,\mathcal{F},\PP,(\theta_t)_{t\in\mathbb{R}})$} is a
$(\mathcal{B}(\mathbb{R}^+)\times\mathcal{F}\times\mathcal{B}(X),\mathcal{B}(X))$-measurable
mapping
\[
\begin{split}\varphi:\mathbb{R}^{+}\times\Omega\times X\rightarrow
X,\quad (t,\omega,x)\mapsto\varphi(t,\omega,x)
\end{split}
\]
such that the following properties hold:
\begin{itemize}
\vspace{-0.1in}\item[(1)] $\varphi(0,\omega,x)=x$ for all $\omega\in\Omega$;

\vspace{-0.1in}\item[(2)]
$\varphi(t+s,\omega,\cdot)=\varphi(t,\theta_s\omega,\varphi(s,\omega,\cdot))$
for all $s,t\geq0$ and $\omega\in\Omega$;

\vspace{-0.1in}\item[(3)] $\varphi(t,\omega,x)$ is continuous in $x$ for every $t\geq 0$ and $\omega\in\Omega$.
\end{itemize}
\end{definition}

For given $x\in X$ and $E,F\subset X$, we define
\begin{equation*}
d(x,F)=\inf_{y\in F}d(x,y)
\end{equation*}
and
\begin{equation*}
d_H(E,F)=\sup_{x\in E}d(x,F).
\end{equation*}
$d_H(E,F)$ is called the {\it  Hausdorff semi-distance} from $E$ to $F$.

\begin{definition}\label{random-set-attractor}
\begin{itemize}
\item[(1)] A set-valued mapping $\omega\mapsto
D(\omega):\Omega\rightarrow 2^{X}$ is said to be a {\rm random set} if the
mapping $\omega\mapsto d(x,D(\omega))$ is measurable for any $x\in
X$. If $\omega\mapsto d(x,D(\omega))$ is measurable for any $x\in
X$ and $D(\omega)$ is closed (compact) for each $\omega\in\Omega$,
then $\omega\mapsto D(\omega)$ is called a {\rm random closed (compact)
set}. A random set $\omega\mapsto D(\omega)$ is said to be {\rm bounded} if
there exist $x_0\in X$ and a random variable $R(\omega)>0$ such that
\begin{equation*}
D(\omega)\subset\{x\in X:d(x,x_0)\leq R(\omega)\}\quad\text{for
all}\quad \omega\in\Omega.
\end{equation*}

\vspace{-0.1in}\item[(2)] A random set $\omega\mapsto D(\omega)$ is called {\rm tempered}
provided that for some $x_0\in X$ and  $\PP$-a.e. $\omega\in\Omega$,
\begin{equation*}
\lim\limits_{t\rightarrow\infty}e^{-\beta t}\sup\{d(b,x_0):b\in
D(\theta_{-t}\omega)\}=0\quad\text{for all}\quad\beta>0.
\end{equation*}

\vspace{-0.1in}\item[(3)] A random set $\omega\mapsto B(\omega)$ is said to be a
{\rm random absorbing set} if for any tempered random set $\omega\mapsto
D(\omega)$, there exists $t_0(\omega)$ such that
\begin{equation*}
\varphi(t,\theta_{-t}\omega,D(\theta_{-t}\omega))\subset
B(\omega)\quad\text{for all}\quad t\geq
t_0(\omega),\,\,\omega\in\Omega.
\end{equation*}

\vspace{-0.1in}\item[(4)] A random set $\omega\mapsto B_1(\omega)$ is said to be a
{\rm random attracting set} if for any tempered random set $\omega\mapsto
D(\omega)$, we have
\begin{equation*}
\lim_{t\rightarrow\infty}d_{H}(\varphi(t,\theta_{-t}\omega,D(\theta_{-t}\omega),B_1(\omega))=0\quad\text{for
all}\quad\omega\in\Omega.
\end{equation*}

\item[(5)] A random compact set $\omega\mapsto A(\omega)$ is said to be
a {\rm global random attractor}  if it is a random attracting set and
$\varphi(t,\omega,A(\omega))=A(\theta_t\omega)$ for all
$\omega\in\Omega$ and $t\geq 0$.
\end{itemize}
\end{definition}

\begin{theorem}\label{existence-random-attractor1}
Let $\varphi$ be a continuous random dynamical system over
$(\Omega,\mathcal{F},\PP,(\theta_t)_{t\in\mathbb{R}})$.  If there is
a random compact attracting set $\omega\mapsto B(\omega)$ of
$\varphi$, then $\omega\mapsto A(\omega)$ is a global random attractor of
$\varphi$, where
\begin{equation*}
A(\omega)=\bigcap_{t>0}\overline{\bigcup_{\tau\geq
t}\varphi(\tau,\theta_{-\tau}\omega,B(\theta_{-\tau}\omega))},\quad\omega\in\Omega.
\end{equation*}
\end{theorem}

\begin{proof}
See \cite{A, C}.
\end{proof}

\section{Basic Settings}

In this section, we give some basic settings about
\eqref{main-eq} and show that it generates a random dynamical
system.

First,
let $Y=(u,v)^{\top}$ and $F(\theta_t\omega,Y)=(z(\theta_t\omega),f-\beta
g(u)+(1-\alpha)z(\theta_t\omega))^{\top}$.  System \eqref{RDE} can then be written as
\begin{equation}\label{matrixRDE}
\dot{Y}=CY+F(\theta_t\omega,Y),
\end{equation}
where $C$ is as in \eqref{c-eq}.

Recall that $z(\theta_t\omega)=(z_j(\theta_t\omega))_{j\in\mathbb{Z}_N^d}$
is the unique stationary solution of \eqref{OU-eq}. Note
that the random variable $|z_j(\omega)|$ is tempered and the mapping
$t\mapsto z_j(\theta_t\omega)$ is $\PP$-a.s. continuous (see
\cite{A, BLW}). More precisely, there is a $\theta_t$-invariant
$\tilde{\Omega}\subset\Omega$ with $\PP(\tilde{\Omega})=1$ such that
$t\mapsto z_j(\theta_t\omega)$ is continuous for
$\omega\in\tilde{\Omega}$ and $j\in\mathbb{Z}_N^d$.
We will consider \eqref{RDE} or \eqref{matrixRDE} for
$\omega\in\tilde{\Omega}$ and write $\tilde{\Omega}$ as $\Omega$
from now on.

Let $E=\mathbb{R}^{N^d}\times\mathbb{R}^{N^d}$ and
$F^{\omega}(t,Y):=F(\theta_t\omega,Y)$, then
$F^{\omega}(\cdot,\cdot):\mathbb{R}\times E\rightarrow E$ is
continuous in $t$ and globally Lipschitz continuous in $Y$ for each
$\omega\in\Omega$. By classical theory of ordinary differential
equations concerning existence and uniqueness of solutions, for each
$\omega\in\Omega$ and any $Y_0\in E$, \eqref{matrixRDE} has a
uniqueness solution $Y(t,\omega,Y_0)$, $t\geq0$, satisfying
\begin{equation}\label{integral-equa}
Y(t,\omega,Y_0)=e^{Ct}Y_0+\int_{0}^{t}e^{C(t-s)}F(\theta_s\omega,Y(s,\omega,Y_0))ds,\quad
t\geq0.
\end{equation}
Moreover, it follows from \cite{A} that $Y(t,\omega,Y_0)$ is measurable in $(t,\omega,Y_0)$.
Hence
\eqref{matrixRDE} generates a continuous random dynamical system on $E$,
\begin{equation}\label{RDE-RDS}
Y:\mathbb{R}^{+}\times\Omega\times E\rightarrow E,\quad
(t,\omega,Y_0)\mapsto Y(t,\omega,Y_0).
\end{equation}
Define a mapping $\phi:\mathbb{R}^{+}\times\Omega\times E\rightarrow
E$ by
\begin{equation}\label{SDE-RDS}
\phi(t,\omega,\phi_0)=Y(t,\omega,Y_0(\omega))+(0,z(\theta_t\omega))^{\top},
\end{equation}
where $\phi_0=(u_0,u_1)^{\top}\in E$ and
$Y_0(\omega)=(u_0,u_1-z(\omega))^{\top}$. Then $\phi$ is a
continuous random dynamical system associated with the problem
\eqref{main-eq} on $E$.

Recall that the eigenvalues of $C$ are given by (see \cite{LZ,QQZ}
for example)
\begin{equation}\label{eigenvalueC-1}
\mu_{i}^{\pm}=\frac{-\alpha\pm\sqrt{\alpha^2-4K\lambda_i}}{2},\quad
i=0,1,\dots,N^d-1.
\end{equation}
By \eqref{eigenvalueC-1}, $C$ has at least two real eigenvalues $0$
and $-\alpha$ with eigenvalues
$\eta_0=(1,\dots,1,0,\dots,0)^{\top}$,
$\eta_{-1}=(1,\dots,1,-\alpha,\dots,-\alpha)^{\top}\in E$,
respectively. Let $E_1=\text{span}\{\eta_0\}$,
$E_{-1}=\text{span}\{\eta_{-1}\}$, $E_{11}=E_{1}+E_{-1}$ and
$E_{22}=E_{11}^{\bot}$, the orthogonal complement space of $E_{11}$
in $E$, then $E=E_{11}\oplus E_{22}$. To control the
unboundedness  of solutions in the direction of $\eta_0$, we will
study \eqref{matrixRDE} in the cylindrical space $E_{1}/{\kappa
\eta_0\mathbb{Z}}\times E_{2}$, where  $E_2=E_{-1}\oplus E_{22}$ (see Section 4 for details).

Observe that the Lipschitz constant of $F$ with respect to  $Y$ in $E$ with the standard Euclidean norm
is independent of $\alpha>0$. But $\mu_i^+\to 0$ as $\alpha\to\infty$ for $i\geq 1$, which gives rise to
some difficulty for the investigation of \eqref{matrixRDE} in $E$ with  the standard
Euclidean norm. To overcome the difficulty,
 we introduce a new norm which is equivalent to the standard
Euclidean norm on $E$. Here, we only collect some results about the
new norm (see \cite{Mor, QQZ} for details).

Define two bilinear forms on $E_{11}$ and $E_{22}$, respectively. For
$Y_i=(u_i,v_i)^{\top}\in E_{11}$, $i=1,2$, let
\begin{equation}\label{inner-E11}
\langle Y_1,Y_2\rangle_{E_{11}}=\frac{\alpha^2}{4}\langle
u_1,u_2\rangle+\langle\frac{\alpha}{2}u_1+v_1,\frac{\alpha}{2}u_2+v_2\rangle,
\end{equation}
where $\langle\cdot,\cdot\rangle$ denotes the inner product on
$\mathbb{R}^{N^d}$, and for $Y_i=(u_i,v_i)^{\top}\in
E_{22},\,i=1,2$, let
\begin{equation}\label{inner-E22}
\langle Y_1,Y_2\rangle_{E_{22}}=\langle
KAu_1,u_2\rangle+(\frac{\alpha^2}{4}-\delta K\lambda_1)\langle
u_1,u_2\rangle+\langle\frac{\alpha}{2}u_1+v_1,\frac{\alpha}{2}u_2+v_2\rangle,
\end{equation}
where $\delta\in(0,1]$. It is easy to check that the
Poincar\'{e}-type inequality
\[
\begin{split}
\langle Au,u\rangle\geq\lambda_1\|u\|^2,\quad\forall\,\,
Y=(u,v)^{\top}\in E_{22}
\end{split}
\]
holds (see \cite{QQZ} for example), where $\|\cdot\|$ is the standard Euclidean norm
on $\mathbb{R}^{N^d}$. Thus \eqref{inner-E22} is
positive definite. For any $Y_i=Y^{(1)}_i+Y^{(2)}_i\in E$,
$i=1,2$, where $Y^{(1)}_1, Y^{(1)}_2\in E_{11}$, $Y^{(2)}_1,
Y^{(2)}_2\in E_{22}$, we define
\begin{equation}\label{inner-E}
\langle Y_1,Y_2\rangle_{E}=\langle
Y^{(1)}_1,Y^{(1)}_2\rangle_{E_{11}}+\langle
Y^{(2)}_1,Y^{(2)}_2\rangle_{E_{22}}.
\end{equation}

\begin{lemma}[\cite{QQZ}]\label{lemma-4}
\begin{itemize}
\item[(1)] \eqref{inner-E11} and \eqref{inner-E22} define inner
products on $E_{11}$ and $E_{22}$, respectively.

\item[(2)] \eqref{inner-E} defines an inner product on $E$, and the
corresponding norm $\|\cdot\|_{E}$ is equivalent to the standard
Euclidean norm on $E$.

\item[(3)] In terms of the inner product $\langle\cdot,\cdot\rangle_{E}$,
$E_1$ and $E_{11}$ are orthogonal to $E_{-1}$ and $E_{22}$,
respectively.

\item[(4)] In terms of the norm $\|\cdot\|_{E}$, the Lipschitz constant
$L_{F}$ of $F$ with respect to $Y$ satisfies
\begin{equation}\label{lipschitz-constant}
L_{F}=\frac{2c_2|\beta|}{\alpha},
\end{equation}
where $c_2$ is as in {\bf (HG)}.
\end{itemize}
\end{lemma}

Note that $E_2$ is orthogonal to $E_1$
and $E=E_1\oplus E_2$. Denote by $P$ and $Q\,(=I-P)$ the projections
from $E$ into $E_1$ and $E_2$, respectively. Set
\begin{equation}\label{a-eq}
a=\frac{\alpha}{2}-\Big|\frac{\alpha}{2}-\frac{\delta
K\lambda_1}{\alpha}\Big|.
\end{equation}

\begin{lemma}\label{lemma-3}
\begin{itemize}
\item[(1)] For any $Y\in E_2$, $\langle
CY,Y\rangle_{E}\leq-a\|Y\|_{E}^{2}$.

\item[(2)] $\|e^{Ct}Q\|_E\leq e^{-at}$ for $t\geq0$.

\item[(3)] $e^{Ct}PY=PY$ for $Y\in E$, $t\geq0$.
\end{itemize}
\end{lemma}
\begin{proof}
(1) and (2) follow from
similar arguments as in Lemma 2.3 and Corollary 2.4 in \cite{QQWZ}. Let us show (3).  For $Y\in E$, since $PY\in E_1$ and
$\frac{d}{dt}e^{Ct}PY=e^{Ct}CPY=0$, we have $e^{Ct}PY=e^{C0}PY=PY$.
\end{proof}

By Lemma \ref{lemma-3} (2), the constant $a$ in \eqref{a-eq} describes the exponential decay rate of
$e^{Ct}|_{QE}$ in the new norm.
By Lemma \ref{lemma-4} (4),  $L_F$ tends to $0$ as $\alpha\to\infty$ with respect
to the new norm, which essentially helps to overcome the difficulty
induced from  the fact that $\mu_i^+\to 0$ as $\alpha\to 0$ for
$i\geq 1$.

The following lemma will be needed to take care unboundedness of the solutions
along the direction of the eigenvectors corresponding to $\mu_0^+$.

\begin{lemma}\label{translation-invariant}
Let $p_0=\kappa\eta_0\in E$ ($\kappa$ is the smallest
positive period of $g$). The random dynamical system $Y$ defined in
\eqref{RDE-RDS} is $p_0$-translation invariant in the sense that
\[
\begin{split}
Y(t,\omega,Y_0+p_0)=Y(t,\omega,Y_0)+p_0,\quad
t\geq0,\,\,\omega\in\Omega,\,\,Y_0\in E.
\end{split}
\]
\end{lemma}
\begin{proof}
Since $Cp_0=0$ and $F(t,\omega,Y)$ is $p_0$-periodic in $Y$,
$Y(t,\omega,Y_0)+p_0$ is a solution of \eqref{matrixRDE} with
initial data $Y_0+p_0$. Thus,
$Y(t,\omega,Y_0)+p_0=Y(t,\omega,Y_0+p_0)$.
\end{proof}
By \eqref{RDE-RDS} and Lemma \ref{translation-invariant}, $\phi$ is
also $p_0$-translation invariant.

\begin{lemma}\label{tempered-random-variable}
For any $\epsilon>0$, there is tempered random variable
$\tilde{r}(\omega)>0$ such that
\begin{equation}
\begin{split}
\|z(\theta_{t}\omega)\|\leq
e^{\epsilon|t|}\tilde{r}(\omega)\quad\text{for all}\quad
t\in\mathbb{R},\,\,\omega\in\Omega,
\end{split}\label{2j}
\end{equation}
where $\tilde{r}(\omega)$ satisfies
\begin{equation}
\begin{split}
e^{-\epsilon|t|}\tilde{r}(\omega)\leq\tilde{r}(\theta_t\omega)\leq
e^{\epsilon|t|}\tilde{r}(\omega),\quad
t\in\mathbb{R},\,\,\omega\in\Omega.
\end{split}\label{2k}
\end{equation}
\end{lemma}
\begin{proof}
For $j\in\mathbb{Z}_N^d$, since $|z_j(\omega)|$ is a tempered random
variable and the mapping $t\mapsto \ln |z_j(\theta_t\omega)|$ is
$\PP$-a.s. continuous, it follow from Proposition 4.3.3 in \cite{A}
that for any $\epsilon_j>0$ there is an tempered random variable
$r_j(\omega)>0$ such that
\[
\begin{split}
\frac{1}{r_j(\omega)}\leq|z_j(\omega)|\leq r_j(\omega),
\end{split}
\]
where $r_j(\omega)$ satisfies, for $\PP$-a.e. $\omega\in\Omega$,
\begin{equation}
\begin{split}
e^{-\epsilon_j|t|}r_j(\omega)\leq r_j(\theta_t\omega)\leq
e^{\epsilon_j|t|}r_j(\omega),\quad t\in\mathbb{R}.
\end{split}\label{2l}
\end{equation}

Let $r(\omega)=(r_j(\omega))_{j\in\mathbb{Z}_N^d}$,
$\omega\in\Omega$ and take $\epsilon_j=\epsilon$,
$j\in\mathbb{Z}_N^d$, then we have
\[
\begin{split}
\|z(\theta_t\omega)\|\leq\Bigg(\sum\limits_{j\in\mathbb{Z}_N^d}e^{2\epsilon|t|}r^2_j(\omega)\Bigg)^{\frac{1}{2}}=e^{\epsilon|t|}\|r(\omega)\|,\quad
t\in\mathbb{R},\,\,\omega\in\Omega.
\end{split}
\]
Let $\tilde{r}(\omega)=\|r(\omega)\|$, $\omega\in\Omega$. Then \eqref{2j}
is satisfied  and \eqref{2k} is trivial from \eqref{2l}.
\end{proof}

\section{Existence of Random Attractor}
In this section, we study the existence of a random attractor. We
assume that $p_0=\kappa\eta_0\in E_{1}$ and $\delta\in (0,1]$ is
such that $a>0$, where $a$ is as in \eqref{a-eq}. We remark in the
end of this section that such $\delta$ always exists.

By Lemma \ref{translation-invariant} and the fact that $C$ has a
zero eigenvalue, we will define a random dynamical system
$\mathbf{Y}$  on some cylindrical space  induced from the random dynamical system  $Y$ on $E$. Then by properties
of $Y$ restricted on $E_2$, we can prove the existence of a global random
attractor of $\mathbf{Y}$. Thus, we can say that $Y$ has a global
random attractor which is unbounded along $E_1$ and bounded along $E_2$. Now, we define $\mathbf{Y}$.

Let $\mathbb{T}^{1}=E_{1}/{p_0\mathbb{Z}}$ and
$\mathbf{E}=\mathbb{T}^{1}\times E_{2}$, where
$p_0\mathbb{Z}=\{kp_0:k\in\mathbb{Z}\}$. For $Y_0\in E$, let
$\mathbf{Y_0}:=Y_0\,\,(mod\,p_0)$, which is an element of $\mathbf{E}$.
Note that, by Lemma \ref{translation-invariant},
$Y(t,\omega,Y_0+kp_0)=Y(t,\omega,Y_0)+kp_0,\,\,\forall
k\in\mathbb{Z}$ for $t\geq0$, $\omega\in\Omega$ and $Y_0\in E$. With
this, we define
$\mathbf{Y}:\mathbb{R}^{+}\times\Omega\times\mathbf{E}\rightarrow\mathbf{E}$
by setting
\begin{equation}\label{induced-RDE-RDS}
\mathbf{Y}(t,\omega,\mathbf{Y_0})=Y(t,\omega,Y_0)\,\,(mod\,p_0),
\end{equation}
where $\mathbf{Y_0}=Y_0\,\,(mod\,p_0)$. Then
$\mathbf{Y}:\mathbb{R}^{+}\times\Omega\times\mathbf{E}\rightarrow\mathbf{E}$
is a random dynamical system. Similarly, the random dynamical system
$\phi$ defined in \eqref{SDE-RDS} also induces a random dynamical
system $\mathbf{\Phi}$ on $\mathbf{E}$. By \eqref{RDE-RDS},
\eqref{SDE-RDS} and \eqref{induced-RDE-RDS}, $\mathbf{\Phi}$ is
defined by
\begin{equation}\label{induced-SDE-RDS}
\mathbf{\Phi}(t,\omega,\mathbf{\Phi_0})=\mathbf{Y}(t,\omega,\mathbf{Y_0})+\tilde{z}(\theta_t\omega)\,\,(mod\,p_0),\quad
t\geq0,\,\,\omega\in\Omega,
\end{equation}
where $\mathbf{\Phi_0}=\phi_0\,\,(mod\,p_0)$,
$\tilde{z}(\theta_t\omega)=(0,z(\theta_t\omega))^{\top}$ and
$\mathbf{Y_0}=\mathbf{\Phi_0}-\tilde{z}(\omega)\,\,(mod\,p_0)$.

Recall that  $P$ and $Q\,(=I-P)$ are the projections
from $E$ into $E_1$ and $E_2$, respectively.

\begin{definition}\label{pseudo-ball-tempered}
Let $\omega\in\Omega$ and $R:\Omega\rightarrow\mathbb{R}^{+}$ be a
random variable. A {\rm random pseudo-ball $\omega\mapsto B(\omega)$ in
$E$ with random radius $\omega\mapsto R(\omega)$} is a set of the
form
\[
\begin{split}
\omega\mapsto B(\omega)=\{b\in E:\|Qb\|_E\leq
R(\omega)\}.
\end{split}
\]
Furthermore, a random set $\omega\mapsto B(\omega)\in E$ is called
{\rm
pseudo-tempered} provided that $\omega\mapsto QB(\omega)$ is a tempered
random set in $E$, i.e., for $\PP$-a.e.$\omega\in\Omega$,
\[
\begin{split}
\lim\limits_{t\rightarrow\infty}e^{-\beta t}\sup\{\|Qb\|_{E}:b\in
B(\theta_{-t}\omega)\}=0\quad\text{for all}\quad\beta>0.
\end{split}
\]
\end{definition}
Clearly, any random pseudo-ball $\omega\mapsto B(\omega)$ in $E$ has
the form $\omega\mapsto E_1\times QB(\omega)$, where $\omega\mapsto
QB(\omega)$ is a random ball in $E_2$. Then the measurability of
$\omega\mapsto B(\omega)$ is trivial. By Definition
\ref{pseudo-ball-tempered}, if $\omega\mapsto B(\omega)$ is a random
pseudo-ball in $E$, then $\omega\mapsto B(\omega)\,\,(mod\,p_0)$ is
random bounded set in $\mathbf{E}$. And if $\omega\mapsto B(\omega)$
is a pseudo-tempered random set in $E$, then $\omega\mapsto
B(\omega)\,\,(mod\,p_0)$ is a tempered random set in $\mathbf{E}$.

We next show the existence of a global random attractor of the induced
random dynamical system $\mathbf{Y}$ defined in
\eqref{induced-RDE-RDS}.
\begin{theorem}\label{existence-random-attractor2}
Let $\alpha>0$ and $K>0$. Then the induced random dynamical system
$\mathbf{Y}$ defined in \eqref{induced-RDE-RDS} has a global random
attractor $\omega\mapsto\mathbf{A_0}(\omega)$.
\end{theorem}
\begin{proof}
For $\omega\in\Omega$, we obtain from \eqref{integral-equa} that
\begin{equation}\label{3a}
Y(t,\omega,Y_0(\omega))=e^{Ct}Y_0(\omega)+\int_{0}^{t}e^{C(t-s)}F(\theta_s\omega,Y(s,\omega,Y_0(\omega)))ds.
\end{equation}
The projection of \eqref{3a} on $E_2$ is
\begin{equation}\label{3b}
QY(t,\omega,Y_0(\omega))=e^{Ct}QY_0(\omega)+\int_{0}^{t}e^{C(t-s)}QF(\theta_s\omega,Y(s,\omega,Y_0(\omega)))ds.
\end{equation}
By replacing $\omega$ by $\theta_{-t}\omega$, it follows from
\eqref{3b} that
\begin{equation*}
QY(t,\theta_{-t}\omega,Y_0(\theta_{-t}\omega))=e^{Ct}QY_0(\theta_{-t}\omega)+\int_{0}^{t}e^{C(t-s)}QF(\theta_{s-t}\omega,Y(s,\theta_{-t}\omega,Y_0(\theta_{-t}\omega)))ds.
\end{equation*}
It then follows from Lemma \ref{lemma-3} and $Q^2=Q$ that
\begin{equation}
\begin{split}
&\|QY(t,\theta_{-t}\omega,Y_0(\theta_{-t}\omega))\|_{E}\\
&\quad\quad\leq
e^{-at}\|QY_0(\theta_{-t}\omega)\|_{E}+\int_{0}^{t}e^{-a(t-s)}\|F(\theta_{s-t}\omega,Y(s,\theta_{-t}\omega,Y_0(\theta_{-t}\omega)))\|_{E}ds.
\end{split}\label{3c}
\end{equation}
By Lemma \ref{tempered-random-variable} with $\epsilon=\frac{a}{2}$
and the equivalence of $\|\cdot\|_{E}$ and $\|\cdot\|$ on $E$, there
is a $M_1>0$ such that
\begin{equation*}
\begin{split}
&\|F(\theta_{s-t}\omega,Y(s,\theta_{-t}\omega,Y_0(\theta_{-t}\omega)))\|_{E}\\
&\quad\quad\leq M_1\|F(\theta_{s-t}\omega,Y(s,\theta_{-t}\omega,Y_0(\theta_{-t}\omega)))\|\\
&\quad\quad=M_1\Big(\|z(\theta_{s-t}\omega)\|^2+\|f-\beta g(Y_u(s,\theta_{-t}\omega))+(1-\alpha)z(\theta_{s-t}\omega)\|^2\Big)^{\frac{1}{2}}\\
&\quad\quad\leq M_1\Big((3\alpha^2-6\alpha+4)\|z(\theta_{s-t}\omega)\|^2+3\|f\|^2+3\beta^2c_1^2N^d\Big)^{\frac{1}{2}}\\
&\quad\quad\leq a_1e^{\frac{a}{2}(t-s)}\tilde{r}(\omega)+a_2,
\end{split}
\end{equation*}
where $Y_u$ satisfies
$Y(s,\theta_{-t}\omega,Y_0(\theta_{-t}\omega))=(Y_u(s,\theta_{-t}\omega),Y_v(s,\theta_{-t}\omega))^{\top}$,
$a_1=M_1\sqrt{3\alpha^2-6\alpha+4}$ and
$a_2=M_1\sqrt{3\|f\|^2+3\beta^2c_1^2N^d}$. We then find from
\eqref{3c} that
\begin{equation*}
\|QY(t,\theta_{-t}\omega,Y_0(\theta_{-t}\omega))\|_{E}\leq
e^{-at}\|QY_0(\theta_{-t}\omega)\|_{E}+\frac{2a_1}{a}(1-e^{-\frac{a}{2}t})\tilde{r}(\omega)+\frac{a_2}{a}(1-e^{-at}).
\end{equation*}
Now for $\omega\in\Omega$, define
\begin{equation*}
R_0(\omega)=\frac{4a_1}{a}\tilde{r}(\omega)+\frac{2a_2}{a}.
\end{equation*}
Then, for any pseudo-tempered random set $\omega\mapsto B(\omega)$
in $E$ and any $Y_0(\theta_{-t}\omega)\in B(\theta_{-t}\omega)$,
there is a $T_{B}(\omega)>0$ such that for $t\geq T_{B}(\omega)$,
\begin{equation*}
\|QY(t,\theta_{-t}\omega,Y_0(\theta_{-t}\omega))\|_{E}\leq
R_0(\omega),\,\,\omega\in\Omega,
\end{equation*}
which implies
\begin{equation*}
Y(t,\theta_{-t}\omega,B(\theta_{-t}\omega))\subset B_0(\omega)\quad
\text{for all}\quad t\geq T_{B}(\omega),\,\,\omega\in\Omega,
\end{equation*}
where $\omega\mapsto B_0(\omega)$ is the random pseudo-ball centered
at origin with random radius $\omega\mapsto R_0(\omega)$. Note that
$\omega\mapsto R_0(\omega)$ is a random variable since
$\omega\mapsto\tilde{r}(\omega)$ is a random variable, then the
measurability of random pseudo-ball $\omega\mapsto B_0(\omega)$ is
trivial from Definition \ref{pseudo-ball-tempered}.

For $\omega\in\Omega$, let
$\mathbf{B}(\omega)=B(\omega)\,\,(mod\,p_0)$ and
$\mathbf{B_0}(\omega)=B_0(\omega)\,\,(mod\,p_0)$, we then have
\begin{equation*}
\mathbf{Y}(t,\theta_{-t}\omega,\mathbf{B}(\theta_{-t}\omega))\subset
\mathbf{B_0}(\omega)\quad \text{for all}\quad t\geq
T_{\mathbf{B}}(\omega),\,\,\omega\in\Omega,
\end{equation*}
where $T_{\mathbf{B}}(\omega)=T_{B}(\omega)$ for $\omega\in\Omega$,
i.e., $\omega\mapsto \mathbf{B_0}(\omega)$ is the random absorbing
set of $\mathbf{Y}$. Moreover, $\omega\mapsto \mathbf{B_0}(\omega)$
is bounded and closed, hence compact in $\mathbf{E}$, it then
follows from Theorem \ref{existence-random-attractor1} that
$\mathbf{Y}$ has a global random attractor
$\omega\mapsto\mathbf{A_0}(\omega)$, where
\begin{equation*}
\mathbf{A_0}(\omega)=\bigcap_{t>0}\overline{\bigcup_{\tau\geq
t}\mathbf{Y}(\tau,\theta_{-\tau}\omega,\mathbf{B_0}(\theta_{-\tau}\omega))},\quad\omega\in\Omega.
\end{equation*}
This completes the proof.
\end{proof}

\begin{corollary}\label{existence-random-attractor-corollary}
Let $\alpha>0$ and $K>0$. Then the induced random dynamical system
$\mathbf{\Phi}$ defined in \eqref{induced-SDE-RDS} has a global random
attractor $\omega\mapsto \mathbf{A}(\omega)$, where
$\mathbf{A}(\omega)=\mathbf{A_0}(\omega)+\tilde{z}(\omega)\,\,(mod\,p_0)$
for all $\omega\in\Omega$.
\end{corollary}
\begin{proof}
It follows from \eqref{induced-SDE-RDS} and Theorem
\ref{existence-random-attractor2}.
\end{proof}

\begin{remark}\label{existence-attractor-rk}
\begin{itemize}
\item[(1)] For any $\alpha>0$ and
$K>0$, there is a $\delta\in (0,1]$ such that $\alpha^2>2\delta
K\lambda_1$ which implies $a>0$, where $a$ is as in \eqref{a-eq} and
$\lambda_1$ is the smallest positive eigenvalue of $A$.

\vspace{-0.1in}\item[(2)] We  say that the random dynamical system $Y$(or $\phi$) has a global
random attractor in the sense that the induced random dynamical
system $\mathbf{Y}$(or $\mathbf{\Phi}$) has a global random attractor, and
we will say that $Y$(or $\phi$) has a global random attractor directly in
the sequel. We denote the global random attractor of $Y$ and $\phi$ by
$\omega\mapsto A_0(\omega)$ and $\omega\mapsto A(\omega)$
respectively. Indeed, $\omega\mapsto A_0(\omega)$ and $\omega\mapsto
A(\omega)$ satisfy
\[
\begin{split}
\mathbf{A_0}(\omega)=A_0(\omega)\,\,(mod\,p_0),\quad
\mathbf{A}(\omega)=A(\omega)\,\,(mod\,p_0),\quad\omega\in\Omega.
\end{split}
\]
Hence a global random attractor of $Y$ (or $\phi$) is unbounded along the one-dimensional space $E_1$ and
bounded along the one-codimensional space $E_2$.

\vspace{-0.1in}\item[(3)] Observe the global attractors of many dissipative systems
related to \eqref{main-eq}  is one-dimensional (see
\cite{CSZ,QSZ1,QSZ2,QQWZ,QQZ,QZZ,S,SZS,WaZh1}). Similarly, we expect
that the random attractor $\omega\mapsto A(\omega)$ of $\phi$ is
one-dimensional for each $\omega\in\Omega$ provided that $\alpha$ is
sufficiently large. We prove that this is true in the next section.

\vspace{-0.1in}\item[(4)] By (2), the system \eqref{main-eq} is dissipative along $E_2$
(i.e. it possesses a global random attractor which is bounded along $E_2$). In  section 6, we will show
that \eqref{main-eq} with sufficiently large $\alpha$ and $K$ also
has a rotation number and hence all the solutions tend to oscillate
with the same frequency eventually.
\end{itemize}
\end{remark}

\section{One-dimensional Random Attractor}
In this section, we apply the invariant and inertial manifold theory, in particular,   the theory established in \cite{CLL} to
show that the random attractor of the random dynamical system $\phi$
generated by \eqref{main-eq} is one-dimensional (more precisely, is a horizontal curve)
provided that
$\alpha$ and $K$ are sufficiently large (see Remark
\ref{existence-attractor-rk} (2) for the random attractor). This
method has been applied by Chow, Shen and Zhou \cite{CSZ} to systems of
coupled first order noisy oscillators and by Shen, Zhou and Shen \cite{SZS} to
the stochastic damped sine-Gordon equation.
The reader is referred to \cite{BeFl, ChGi} for the theory and application of inertial manifold theory
for stochastic evolution equations.

Assume that
$p_0=\kappa\eta_0$ and $a>4L_F$ (see \eqref{a-eq} for $a$ and
\eqref{lipschitz-constant} for $L_F$).
Note that the condition $a>4L_F$ indicates that the exponential decay rate of
$e^{Ct}|_{QE}$ in the norm $\|\cdot\|_E$ is larger than four times the Lipschitz
constant of $F$ in the norm $\|\cdot\|_E$. It will be seen at the end of this section
that  the condition $a>4L_F$
can be satisfied provided that $\alpha$ and $K$ are sufficiently
large.

\begin{definition}\label{random-horizontal-curve-def}
Suppose $\{\Phi^{\omega}\}_{\omega\in\Omega}$ is a family of maps
from $E_1$ to $E_2$ and $n\in\mathbb{N}$. A family of graphs
$\omega\mapsto \ell(\omega)\equiv\{(p,\Phi^{\omega}(p)):p\in
E_{1}\}$ is said to be a {\rm  random $np_0$-period horizontal
curve} if $\omega\mapsto \ell(\omega)$ is a random set and
$\{\Phi^{\omega}\}_{\omega\in\Omega}$ satisfy the Lipshitz condition
\[
\begin{split}
\|\Phi^{\omega}(p_1)-\Phi^{\omega}(p_2)\|_{E}\leq
\|p_1-p_2\|_{E}\quad\text{for all}\quad p_1,p_2\in
E_1,\,\,\omega\in\Omega
\end{split}
\]
and the periodic condition
\[
\begin{split}
\Phi^{\omega}(p+np_0)=\Phi^{\omega}(p)\quad\text{for all}\quad p\in
E_1,\,\,\omega\in\Omega.
\end{split}
\]
\end{definition}

Note that for any $\omega\in\Omega$, $\ell(\omega)$ is a
deterministic $np_0$-period horizontal curve. When $n=1$, we simply
call it a horizontal curve.

\begin{lemma}\label{invariance-of-rand-hori-curv}
Let $a>4L_F$. Suppose that $\omega\mapsto \ell(\omega)$ is a random
$np_0$-period horizontal curve in $E$. Then, $\omega\mapsto
Y(t,\omega,\ell(\omega))$ is also a random $np_0$-period horizontal
curve in $E$ for all $t>0$. Moreover, $\omega\mapsto
Y(t,\theta_{-t}\omega,\ell(\theta_{-t}\omega))$ is a random
$np_0$-period horizontal curve for all $t>0$.
\end{lemma}

The proof of Lemma 5.2 is similar to that of Lemma 4.2 in \cite{SZS}.
We hence omit it here.

Choose $\gamma\in(0,\frac{a}{2})$ such that
\begin{equation}\label{4c}
\frac{2c_2|\beta|}{\alpha}\Bigg(\frac{1}{\gamma}+\frac{1}{a-2\gamma}\Bigg)<1,
\end{equation}
where $\frac{2c_2|\beta|}{\alpha}$ is the Lipschitz constant of $F$
(see \eqref{lipschitz-constant}). We remark at the end of this
section that such a $\gamma$ exists provided that $\alpha$ and $K$
are sufficiently large. The main result in this section is as
follows.

\begin{theorem}\label{one-dimension-thm}
Assume that $a>4L_F$ and there is $\gamma\in(0,\frac{a}{2})$ such
that \eqref{4c} holds. Then the global random attractor $\omega\mapsto
A_0(\omega)$ of the random dynamical system $Y$ is a random
horizontal curve.
\end{theorem}
\begin{proof}
Since equation \eqref{matrixRDE} can be viewed as a deterministic
system with a random parameter $\omega\in\Omega$, we write it here
as \eqref{matrixRDE}$_{\omega}$ for  $\omega\in\Omega$. We
first show that for any fixed $\omega\in\Omega$,
 \eqref{matrixRDE}$_{\omega}$ has a one-dimensional
attracting invariant manifold $W(\omega)$.

In order to do so, for fixed $\omega\in\Omega$, let
$$
W(\omega)=\{Y_0\in E\,|\, Y(t,\omega,Y_0)\,\,\,\text{exists for}\,\,\, t\leq 0\,\,\,\text{and}\,\,\, \sup_{t\leq 0}\|e^{\gamma t}Y(t,\omega,Y_0)\|_E<\infty\}.
$$
We prove that $W(\omega)$ is a one-dimensional
attracting invariant manifold of \eqref{matrixRDE}$_{\omega}$.

First of all, by the definition of $W(\omega)$,  it is clear that  for any $t\in\RR$,
$$
Y(t,\omega,W(\omega))=W(\theta_t\omega),
$$
that is, $\{W(\omega)\}_{\omega\in\Omega}$ is invariant.
By the variation of constant formula, $Y_0\in W(\omega)$ if and only if there is $\tilde Y(t)$ with
$\tilde Y(0)=Y_0$, $\sup_{t\leq 0}\|e^{\gamma t}\tilde Y(t)\|_E<\infty$,
\begin{equation}\label{centered-manifold-equation}
\tilde{Y}(t)=e^{Ct}\xi+\int_{0}^{t}e^{C(t-s)}PF^{\omega}(s,\tilde{Y}(s))ds+\int_{-\infty}^{t}e^{C(t-s)}QF^{\omega}(s,\tilde{Y}(s))ds,\quad
t\leq0,
\end{equation}
and $Y(t,\omega,Y_0)=\tilde Y(t)$, where
$F^{\omega}(t,Y)=F(\theta_{t}\omega, Y)$ and $\xi=P\tilde{Y}(0)\in E_1$. For $H:(-\infty,0]\rightarrow E$
such that $\sup_{t\leq0}\|e^{\gamma t}H(t)\|_{E}<\infty$, define
\begin{equation*}\label{operator-L}
(LH)(t)=\int_{0}^{t}e^{C(t-s)}PH(s)ds+\int_{-\infty}^{t}e^{C(t-s)}QH(s)ds,
\quad t\leq0.
\end{equation*}
Then
\begin{equation*}
\sup_{t\leq0}\|e^{\gamma
t}(LH)(t)\|_{E}\leq(\frac{1}{\gamma}+\frac{1}{a-\gamma})\sup_{t\leq0}\|e^{\gamma
t}H(t)\|_E\leq\Bigg(\frac{1}{\gamma}+\frac{1}{a-2\gamma}\Bigg)\sup_{t\leq0}\|e^{\gamma
t}H(t)\|_E,
\end{equation*}
which means that $\|L\|\leq\frac{1}{\gamma}+\frac{1}{a-2\gamma}$.
Thus, Theorem 3.3 in \cite{CLL} shows that for any $\xi\in E_1$,
equation \eqref{centered-manifold-equation} has a unique solution
$\tilde{Y}^{\omega}(t,\xi)$ satisfying $\sup_{t\leq0}\|e^{\gamma
t}\tilde{Y}^{\omega}(t,\xi)\|_{E}<\infty$. Let
\begin{equation*}
h(\omega,\xi)=Q\tilde{Y}^{\omega}(0,\xi)=\int_{-\infty}^{0}e^{-Cs}QF^{\omega}(s,\tilde{Y}^{\omega}(s,\xi))ds,\quad
\omega\in\Omega.
\end{equation*}
Then, $$W(\omega)=\{\xi+h(\omega,\xi):\xi\in E_1\}$$  and $W(\omega)$ is a
one dimensional invariant manifold of \eqref{matrixRDE}$_{\omega}$.
Furthermore, for any $\epsilon\in(0,\gamma)$, by Lemma
\ref{tempered-random-variable}, we have
\begin{equation}\label{a-inequality}
\|h(\theta_{-t}\omega,\xi)\|_E\leq\frac{a_1}{a-\epsilon}\tilde{r}(\omega)e^{\epsilon
t}+\frac{a_2}{a},\quad t\geq0,
\end{equation}
where $a_1, a_2$ is the same as in the proof Theorem
\ref{existence-random-attractor2}.

To show the attracting property of $W(\omega)$,   we prove for each given $\omega\in\Omega$ the existence of a stable
foliation $\{W_s(Y_0,\omega):Y_0\in W(\omega)\}$ of the invariant
 manifold $W(\omega)$ of
\eqref{matrixRDE}$_{\omega}$. Consider the following integral
equation
\begin{equation}
\begin{split}
\hat{Y}(t)&=e^{Ct}\eta+\int_{0}^{t}e^{C(t-s)}Q\Big(F^{\omega}(s,\hat{Y}(s)+Y^{\omega}(s,\xi+h(\omega,\xi)))\\
&\quad\quad\quad\quad\quad\quad\quad\quad\quad-F^{\omega}(s,Y^{\omega}(s,\xi+h(\omega,\xi)))\Big)ds\\
&\quad+\int_{\infty}^{t}e^{C(t-s)}P\Big(F^{\omega}(s,\hat{Y}(s)+Y^{\omega}(s,\xi+h(\omega,\xi)))\\
&\quad\quad\quad\quad\quad\quad\quad-F^{\omega}(s,Y^{\omega}(s,\xi+h(\omega,\xi)))\Big)ds,\quad
t\geq0,
\end{split}\label{foliation-equation}
\end{equation}
where $\xi+h(\omega,\xi)\in W(\omega)$, $\eta=Q\hat{Y}(0)\in
E_2$ and
$Y^{\omega}(t,\xi+h(\omega,\xi)):=Y(t,\omega,\xi+h(\omega,\xi))$,
$t\geq0$ is the solution of \eqref{matrixRDE} with initial data
$\xi+h(\omega,\xi)$ for fixed $\omega\in\Omega$. Theorem 3.4 in
\cite{CLL} shows that for any $\xi\in E_1$ and $\eta\in E_2$,
equation \eqref{foliation-equation} has a unique solution
$\hat{Y}^{\omega}(t,\xi,\eta)$ satisfying $\sup_{t\geq0}\|e^{\gamma
t}\hat{Y}^{\omega}(t,\xi,\eta)\|_{E}<\infty$ and for any $\xi\in
E_1$, $\eta_1,\,\eta_2\in E_2$,
\begin{equation}\label{contraction-property}
\sup_{t\geq0}e^{\gamma
t}\|\hat{Y}^{\omega}(t,\xi,\eta_1)-\hat{Y}^{\omega}(t,\xi,\eta_2)\|_{E}\leq
M_2\|\eta_1-\eta_2\|_{E},
\end{equation}
where
$M_2=\frac{1}{1-\frac{2c_2|\beta|}{\alpha}\big(\frac{1}{\gamma}+\frac{1}{a-2\gamma}\big)}$.
Let
\begin{equation*}
\begin{split}
\hat{h}(\omega,\xi,\eta)&=\xi+P\hat{Y}^{\omega}(0,\xi,\eta)\\
&=\xi+\int_{\infty}^{0}e^{-Cs}P\Big(F^{\omega}(s,\hat{Y}^{\omega}(s,\xi,\eta)+Y^{\omega}(s,\xi+h(\omega,\xi)))\\
&\quad\quad\quad\quad\quad\quad\quad-F^{\omega}(s,Y^{\omega}(s,\xi+h(\omega,\xi)))\Big)ds.
\end{split}
\end{equation*}
Then,
$W_{s}(\omega,\xi+h(\omega,\xi))=\{\eta+h(\omega,\xi)+\hat{h}(\omega,\xi,\eta):\eta\in
E_2\}$ is a foliation of $W(\omega)$ at
$\xi+h(\omega,\xi)$.

Observe that
\begin{equation}
\begin{split}
&\hat{Y}^{\omega}(t,\xi,\eta)+Y^{\omega}(t,\xi+h(\omega,\xi))-Y^{\omega}(t,\xi+h(\omega,\xi))\\
&\quad\quad=\hat{Y}^{\omega}(t,\xi,\eta)\\
&\quad\quad=e^{Ct}(\eta+h(\omega,\xi)+\hat{h}(\omega,\xi,\eta)-\xi-h(\omega,\xi))\\
&\quad\quad\quad+\int_{0}^{t}e^{C(t-s)}\Big(F^{\omega}(s,\hat{Y}^{\omega}(s,\xi,\eta)+Y^{\omega}(s,\xi+h(\omega,\xi)))\\
&\quad\quad\quad\quad\quad\quad\quad\quad\quad-F^{\omega}(s,Y^{\omega}(s,\xi+h(\omega,\xi)))\Big)ds
\end{split}\label{compare-equation1}
\end{equation}
and
\begin{equation}
\begin{split}
&Y^{\omega}(t,\eta+h(\omega,\xi)+\hat{h}(\omega,\xi,\eta))-Y^{\omega}(t,\xi+h(\omega,\xi))\\
&\quad\quad=e^{Ct}(\eta+h(\omega,\xi)+\hat{h}(\omega,\xi,\eta)-\xi-h(\omega,\xi))\\
&\quad\quad\quad+\int_{0}^{t}e^{C(t-s)}\Big(F^{\omega}(s,Y^{\omega}(s,\eta+h(\omega,\xi)+\hat{h}(\omega,\xi,\eta)))\\
&\quad\quad\quad\quad\quad\quad\quad\quad\quad-F^{\omega}(s,Y^{\omega}(s,\xi+h(\omega,\xi)))\Big)ds.
\end{split}\label{compare-equation2}
\end{equation}
Comparing \eqref{compare-equation1} with \eqref{compare-equation2},
we find that
\begin{equation}\label{result-from-comparison}
\hat{Y}^{\omega}(t,\xi,\eta)=Y^{\omega}(t,\eta+h(\omega,\xi)+\hat{h}(\omega,\xi,\eta))-Y^{\omega}(t,\xi+h(\omega,\xi)),\quad
t\geq0.
\end{equation}
In addition, if $\eta=0$, then by the uniqueness of solution of
\eqref{foliation-equation}, $\hat{Y}^{\omega}(t,\xi,0)\equiv0$ for
$t\geq0$, which together with \eqref{contraction-property} and
\eqref{result-from-comparison} shows that
\begin{equation}\label{another-contraction-property}
\sup_{t\geq0}e^{\gamma
t}\|Y^{\omega}(t,\eta+h(\omega,\xi)+\hat{h}(\omega,\xi,\eta))-Y^{\omega}(t,\xi+h(\omega,\xi))\|_{E}\leq
M_2\|\eta\|_E
\end{equation}
for any $\xi\in E_1$ and $\eta\in E_2$. Therefore,  $\{W_s(Y_0,\omega):Y_0\in W(\omega)\}$ is a stable foliation of the
invariant
 manifold $W(\omega)$ of
\eqref{matrixRDE}$_{\omega}$ and then $W(\omega)$ is a one-dimensional attracting invariant manifold of \eqref{matrixRDE}$_{\omega}$.

Next we show that $A_0(\omega)=W(\omega)$ and $A_0(\omega)$ is a random horizontal curve.
 Let $\omega\mapsto B(\omega)$ be any
pseudo-tempered random set in $E$. For any $t>0$ and $Y_0\in B(\theta_{-t}\omega)$, there is
$\xi(\theta_{-t}\omega,Y_0)\in E_1$ such that
\begin{equation*}
\begin{split}
Y_0\in
W_{s}(\theta_{-t}\omega,\xi(\theta_{-t}\omega,Y_0)+h(\theta_{-t}\omega,\xi(\theta_{-t}\omega,Y_0))).
\end{split}
\end{equation*}
Let
$\eta(\theta_{-t}\omega)=\sup_{Y_0\in B(\theta_{-t}\omega)}\|QY_0-h(\theta_{-t}\omega,\xi(\theta_{-t}\omega,Y_0))\|_E$.
Then by \eqref{a-inequality} and
\eqref{another-contraction-property},
\begin{equation*}
\begin{split}
&\|Y(t,\theta_{-t}\omega,Y_0)-Y(t,\theta_{-t}\omega,\xi(\theta_{-t}\omega,Y_0)+h(\theta_{-t}\omega,\xi(\theta_{-t}\omega,Y_0)))\|_{E}\\
&\quad\quad\leq M_2e^{-\gamma t}\eta(\theta_{-t}\omega)\\
&\quad\quad\rightarrow0\quad\text{as}\quad t\rightarrow\infty,
\end{split}
\end{equation*}
which implies that
for $\omega\in\Omega$,
\begin{equation*}
\begin{split}
d_{H}(Y(t,\theta_{-t}\omega,B(\theta_{-t}\omega)),W(\omega))\rightarrow0\quad\text{as}\quad
t\rightarrow\infty.
\end{split}
\end{equation*}
Therefore,
$$A_0(\omega)=W(\omega)\,\,\,{\rm  for}\,\,\, \omega\in\Omega.
$$
Moreover, for any random horizontal curve
$\omega\mapsto\ell(\omega)$ in $E$  contained in some
pseudo-tempered random set,
\begin{equation*}
\begin{split}
d_{H}(Y(t,\theta_{-t}\omega,\ell(\theta_{-t}\omega)),A_0(\omega))\rightarrow0\quad\text{as}\quad
t\rightarrow\infty
\end{split}
\end{equation*}
for every $\omega\in\Omega$,
which means that
$\lim_{t\rightarrow\infty}Y(t,\theta_{-t}\omega,\ell(\theta_{-t}\omega))\subset
A_0(\omega)$. Since $A_0(\omega)$ is one-dimensional, we have for
$\omega\in\Omega$,
\begin{equation*}
\begin{split}
A_0(\omega)=\lim_{t\rightarrow\infty}Y(t,\theta_{-t}\omega,\ell(\theta_{-t}\omega)).
\end{split}
\end{equation*}
It then follows from Lemma \ref{invariance-of-rand-hori-curv} that
$\omega\mapsto A_0(\omega)$ is a random horizontal curve.
\end{proof}

\begin{corollary}
\label{one-dimension-cor} Assume that $a>4L_F$ and there is
$\gamma\in(0,\frac{a}{2})$ such that \eqref{4c} holds. Then the
random attractor $\omega\mapsto A(\omega)$ of the random dynamical
system $\phi$ is a random horizontal curve.
\end{corollary}
\begin{proof}
It follows from Corollary
\ref{existence-random-attractor-corollary},  Remark
\ref{existence-attractor-rk}
 and Theorem \ref{one-dimension-thm}.
\end{proof}

\begin{remark}\label{one-dimension-rk}
At the beginning of this section, we assume that $a>4L_{F}$. Since
$a=\frac{\alpha}{2}-|\frac{\alpha}{2}-\frac{\delta K\lambda_1}{\alpha}|$
and $L_{F}=\frac{2c_2|\beta|}{\alpha}$, we can take $\alpha$, $K$
satisfying $\frac{\alpha}{2}-\Big|\frac{\alpha}{2}-\frac{\delta
K\lambda_1}{\alpha}\Big|>\frac{8c_2|\beta|}{\alpha}$, where
$\lambda_1$ is the smallest positive eigenvalue of $A$. On the other
hand, we need some $\gamma\in(0,\frac{a}{2})$ such that \eqref{4c}
holds. Note that
\begin{equation*}
\begin{split}
\min_{\gamma\in(0,\frac{a}{2})}\Bigg(\frac{1}{\gamma}+\frac{1}{a-2\gamma}\Bigg)=\Bigg(\frac{1}{\gamma}
+\frac{1}{a-2\gamma}\Bigg)\Bigg|_{\gamma=\frac{a}{2+\sqrt{2}}}=\frac{(\sqrt{2}+1)^2}{a},
\end{split}
\end{equation*}
which implies that there exist $\alpha$, $K$ satisfying
\begin{equation}\label{4n}
\frac{\alpha}{2}-\Big|\frac{\alpha}{2}-\frac{\delta
K\lambda_1}{\alpha}\Big|>\frac{2c_2|\beta|(\sqrt{2}+1)^2}{\alpha}>\frac{8c_2|\beta|}{\alpha}.
\end{equation}
Indeed, let $c=2c_2|\beta|(\sqrt{2}+1)^2$, then for any
$\alpha>\sqrt{2c}$ and $K>\frac{c}{\lambda_1}$, there is a
$\delta>0$ satisfying
\begin{equation*}
\begin{split}
\frac{c}{K\lambda_1}<\delta<\min\Big\{\frac{\alpha^2-c}{K\lambda_1},1\Big\}
\end{split}
\end{equation*}
such that \eqref{4n} holds.
\end{remark}

\section{Rotation Number}
In this section, we study the phenomenon of frequency locking, i.e.,
the existence of a rotation number of the coupled second order
oscillators with white noises \eqref{main-eq}.

\begin{definition}\label{definition-rotation-number}
The coupled second order system with white noises \eqref{main-eq} is
said to have a {\rm rotation number $\rho\in\mathbb{R}$} if, for
$\PP$-a.e. $\omega\in\Omega$ and each $\phi_0=(u_0,u_1)^{\top}\in
E$, the limit
$\lim_{t\rightarrow\infty}\frac{P\phi(t,\omega,\phi_0)}{t}$ exists
and
\[
\begin{split}
\lim\limits_{t\rightarrow\infty}\frac{P\phi(t,\omega,\phi_0)}{t}=\rho\eta_0,
\end{split}
\]
where $\eta_0$ is the basis of $E_{1}$.
\end{definition}

Note that the rotation number is considered here by restricting
$\phi$ on $E_1$, since $\phi$ is dissipative on $E_2$ and limits
likewise in Definition \ref{definition-rotation-number} vanish. From
\eqref{SDE-RDS}, we have
\begin{equation}
\begin{split}
\frac{P\phi(t,\omega,\phi_0)}{t}=\frac{PY(t,\omega,Y_0(\omega))}{t}+\frac{P(0,z(\theta_t\omega))^{\top}}{t},
\end{split}\label{an-equality}
\end{equation}
where $\phi_0=(u_0,u_1)^{\top}$ and
$Y_0(\omega)=(u_0,u_1-z(\omega))^{\top}$. By Lemma 2.1 in
\cite{DLS}, it is easy to prove that
$\lim_{t\rightarrow\infty}\frac{P(0,z(\theta_t\omega))^{\top}}{t}=(0,0)^{\top}$.
Thus, it sufficient to prove the existence of the rotation number of
the random system \eqref{matrixRDE}.

Let us show a simple lemma which will be used. For any
$p_i=s_i\eta_0\in E_1$, $i=1,2$, we define
\begin{equation*}
p_1\leq p_2\quad\text{if}\quad s_1\leq s_2.
\end{equation*}
Then we have
\begin{lemma}\label{orientation-preserving}
Suppose that $a>4L_F$. Let $\ell$ be any deterministic
$np_0$-periodic horizontal curve ($\ell$ satisfies the Lipschitz and
periodic condition in Definition \ref{random-horizontal-curve-def}).
For any $Y_1,\,\,Y_2\in\ell$ with $PY_1\leq PY_2$, there holds
\begin{equation}\label{orientation inequality}
PY(t,\omega,Y_1)\leq PY(t,\omega,Y_2)\quad\text{for}\,\,
t>0,\,\,\omega\in\Omega.
\end{equation}
\end{lemma}
The proof of this lemma is similar to that of Lemma 6.3 in \cite{SZS}. We
then omit it here. We now have the main result in this section.
\begin{theorem}\label{existence-rotation-number1}
Let $a>4L_F$. Then the rotation number of \eqref{matrixRDE} exists.
\end{theorem}
\begin{proof}
By the random dynamical system $\mathbf{Y}$ defined in \eqref{induced-RDE-RDS},
we define the corresponding skew-product semiflow
$\mathbf{\Theta}_t:\Omega\times\mathbf{E}\rightarrow\Omega\times\mathbf{E}$
for $t\geq0$ by setting
\[
\begin{split}
\mathbf{\Theta}_t(\omega,\mathbf{Y_0})=(\theta_t\omega,\mathbf{Y}(t,\omega,\mathbf{Y_0})).
\end{split}
\]
Obviously,
$(\Omega\times\mathbf{E},\,\mathcal{F}\times\mathcal{B},\,(\mathbf{\Theta}_t)_{t\geq0})$
is a measurable dynamical system, where
$\mathcal{B}=\mathcal{B}(\mathbf{E})$ is the Borel $\sigma$-algebra
of $\mathbf{E}$. It also can be verified that there is a measure
$\mu$ on $\Omega\times\mathbf{E}$ such that
$(\Omega\times\mathbf{E},\,\mathcal{F}\times\mathcal{B},\,\mu,\,(\mathbf{\Theta}_t)_{t\geq0})$
becomes an ergodic metric dynamical system (see \cite{Cr}). Note
that
\[
\begin{split}
\frac{PY(t,\omega,Y_0)}{t}=\frac{PY_0}{t}+\frac{1}{t}\int_{0}^{t}PF(\theta_s\omega,Y(s,\omega,Y_0))ds.
\end{split}
\]
Since
$F(\theta_s\omega,Y(s,\omega,Y_0)+kp_0)=F(\theta_s\omega,Y(s,\omega,Y_0)),\,\,\forall
k\in\mathbb{Z}$, we can identify
$F(\theta_s\omega,\mathbf{Y}(s,\omega,\mathbf{Y_0}))$ with
$F(\theta_s\omega,Y(s,\omega,Y_0))$ and write
\[
\begin{split}
F(\theta_s\omega,Y(s,\omega,Y_0))=F(\theta_s\omega,\mathbf{Y}(s,\omega,\mathbf{Y_0})).
\end{split}
\]
Thus,
\begin{equation}
\begin{split}
\frac{PY(t,\omega,Y_0)}{t}&=\frac{PY_0}{t}+\frac{1}{t}\int_{0}^{t}PF(\theta_s\omega,\mathbf{Y}(s,\omega,\mathbf{Y_0}))ds\\
&=\frac{PY_0}{t}+\frac{1}{t}\int_{0}^{t}\mathbf{F}(\mathbf{\Theta}_s(\omega,\mathbf{Y_0}))ds,\\
\end{split}\label{5b}
\end{equation}
where $\mathbf{F}=P\circ F\in
L^{1}(\Omega\times\mathbf{E},\,\mathcal{F}\times\mathcal{B},\,\mu)$.
Let $t\rightarrow\infty$ in (\ref{5b}),
$\lim_{t\rightarrow\infty}\frac{PY_0}{t}=(0,0)^{\top}$ and by
Ergodic Theorems in \cite{A}, there exist a constant
$\rho\in\mathbb{R}$ such that
\[
\begin{split}
\lim_{t\rightarrow\infty}\frac{1}{t}\int_{0}^{t}\mathbf{F}(\mathbf{\Theta}_s(\omega,\mathbf{Y_0}))ds=\rho\eta_0,
\end{split}
\]
which means
\[
\begin{split}
\lim_{t\rightarrow\infty}\frac{PY(t,\omega,Y_0)}{t}=\rho\eta_0.
\end{split}
\]
for $\mu$-a.e.$(\omega,Y_0)\in\Omega\times E$. Thus, there is
$\Omega^*\subset\Omega$ with $\PP(\Omega^*)=1$ such that for any
$\omega\in\Omega^*$,
 there is $Y_0^{*}(\omega)\in E$ such that
\begin{equation*}
\lim_{t\rightarrow\infty}\frac{PY(t,\omega,Y_0^{*}(\omega))}{t}=\rho\eta_0.
\end{equation*}
By Lemma \ref{translation-invariant}, we have that for any
$n\in\mathbb{N}$ and $\omega\in\Omega^*$,
\begin{equation}\label{5c}
\lim_{t\rightarrow\infty}\frac{PY(t,\omega,Y^{*}_0(\omega)\pm
np_0)}{t}=\lim_{t\rightarrow\infty}\frac{PY(t,\omega,Y^{*}_0(\omega))\pm
np_0}{t}=\rho\eta_0.
\end{equation}

Now for any $\omega\in\Omega^*$ and any $Y\in E$, there is
$n_0(\omega)\in\mathbb{N}$ such that
\begin{equation*}
PY^{*}_0(\omega)-n_0(\omega)p_0\leq PY\leq
PY^{*}_0(\omega)+n_0(\omega)p_0
\end{equation*}
and there is a $n_0(\omega)p_0$-periodic horizontal curve $l_0(
\omega)$ such that $Y^*_0(\omega)-n_0(\omega)p_0$, $Y$,
$Y^*_0(\omega)+n_0(\omega)p_0\in l_0(\omega)$. Then by Lemma
\ref{orientation-preserving}, we have
\begin{equation*}
PY(t,\omega,Y^{*}_0(\omega)-n_0(\omega)p_0)\leq PY(t,\omega,Y)\leq
PY(t,\omega,Y^{*}_0(\omega)+n_0(\omega)p_0),
\end{equation*}
which together with \eqref{5c} implies that for  any
$\omega\in\Omega^*$ and any $Y\in E$,
\begin{equation*}
\lim_{t\rightarrow\infty}\frac{PY(t,\omega,Y)}{t}=\rho\eta_0.
\end{equation*}
Consequently, for any a.e. $\omega\in\Omega$ and any $Y\in E$,
\begin{equation*}
\lim_{t\rightarrow\infty}\frac{PY(t,\omega,Y)}{t}=\rho\eta_0.
\end{equation*}
The theorem is thus proved.
\end{proof}

\begin{corollary}\label{existence-rotation-number1-corollary}
Assume that $a>4L_F$. Then the rotation number of the coupled second
order system with white noises \eqref{main-eq} exists.
\end{corollary}
\begin{proof}
It follows from \eqref{an-equality} and Theorem
\ref{existence-rotation-number1}.
\end{proof}


\begin{thebibliography}{99}
\vspace{-0.08in}\bibitem{A}  L.Arnold, Random dynamical systems, Springer Monographs in Mathematics, Springer-Verlag, Berlin, 1998.

\vspace{-0.08in}\bibitem{BLW} P.W. Bates, K. Lu and B. Wang, Random attractors for stochastic reaction-diffusion equations on unbounded
domains, {\it J. Diff. Eq.} {\bf 246} (2009), 845-869.

\vspace{-0.08in}\bibitem{BeFl} A. Bensoussan, and F. Flandoli,
Stochastic inertial manifold,
{\it Stochastics Stochastics Rep.} {\bf  53} (1995),   13-39.

\vspace{-0.08in}\bibitem{C} I. Chueshov, Monotone random systems theory and applications, Lecture Notes in Mathematics {\bf 1779}, Springer-Verlag, Berlin, 2002.

\vspace{-0.08in}\bibitem{ChGi} I. Chueshov, T. V.  Girya,
 Inertial manifolds and forms for semilinear parabolic equations subjected to additive white noise,
 {\it Lett. Math. Phys.} {\bf 34} (1995),  69-76.

\vspace{-0.08in}\bibitem{Cr} H. Crauel, Random probability measures on Polish spaces, Stochastics Monographs {\bf 11}, Taylor \& Francis, London, 2002.

\vspace{-0.08in}\bibitem{CLL} S.N. Chow, X.B. Lin and K. Lu, Smooth invariant foliations in infinite dimensional
spaces, {\it J. Diff. Eq.} {\bf 94} (1991), 266-291.

\vspace{-0.08in}\bibitem{CSZ} S.N. Chow, W. Shen and H.M. Zhou, Dynamical order in systems of coupled noisy
oscillators, {\it J. Dyn. Diff. Eqns.} {\bf 19} (2007), 1007-1035.

\vspace{-0.08in}\bibitem{DLS} J. Duan, K. Lu and B. Schmalfuss, Invariant manifolds for stochastic partial differential
equations, {\it Ann. Probab.} {\bf 31} (2003), 2109-2135.

\vspace{-0.08in}\bibitem{Fan} X. M. Fan,  Random attractor for a damped sine-Gordon equation with white noise,
{\it  Pacific J. Math.} {\bf 216} (2004), 63-76.

\vspace{-0.08in}\bibitem{HaBeWi} P. Hadley, M.R. Beasley and K. Wiesenfeld, Phase
locking of Josephson junction series arrays, {\it Phys. Rev. B} {\bf
38} (1988), 8712-8719.

\vspace{-0.08in}\bibitem{Hal} J. K. Hale, Asymptotic Behavior of Dissipative Systems, Mathematical Surveys and Monographs
{\bf 25}, American Mathematical Society, Providence, Rhode Island (1988).

\vspace{-0.08in}\bibitem{Lev} M. Levi, Dynamics of discrete Fernkel-Kontorova
models,  in Analysis, Et Cetera, P. Rabinowitz and E. Zehnder, eds.,
Academic Press, New York, 1990.

\vspace{-0.08in}\bibitem{LZ} H.Y. Li and S.F. Zhou, Structure of the global attractor for a second order strongly damped lattice
system, {\it J. Math. Anal. Appl.}  {\bf 330} (2007), 1426-1446.

\vspace{-0.08in}\bibitem{Mor} X. Mora, Finite-dimensional attracting invariant manifolds for damped semilinear wave equations, Res. Notes Math. {\bf 155} (1987), 172-183.

\vspace{-0.08in}\bibitem{QW} M.P. Qian and D. Wang, On a system of hyperstable frequency locking persistence under white
noise, {\it  Ergo. Theo. Dyna. Syst.} {\bf 20} (2000), 547-555.

\vspace{-0.08in}\bibitem{QSZ1} M. Qian, W. Shen and J.Y. Zhang, Global Behavior in the dynamical equation of J-J
type, {\it  J. Diff. Eq.}  {\bf 77} (1988), 315-333.

\vspace{-0.08in}\bibitem{QSZ2} M. Qian, W. Shen and J.Y. Zhang, Dynamical Behavior in the coupled systems of J-J
type, {\it J. Diff. Eq.} {\bf 88} (1990), 175-212.

\vspace{-0.08in}\bibitem{QQWZ} M. Qian, W.X. Qin, G.X. Wang and S. Zhu, Unbounded one dimensional global attractor for the damped
 sine-Gordon equation, {\it J. Nonlinear Sci.} {\bf 10} (2000), 417-432.

\vspace{-0.08in}\bibitem{QQZ} M. Qian, W.X. Qin and S. Zhu, One dimensional global attractor for discretization of the damped
 driven sine-Gordon equation, {\it Nonl. Anal., TMA.} {\bf 34} (1998), 941-951.

\vspace{-0.08in}\bibitem{QZZ} M. Qian, S.F. Zhou and S. Zhu, One dimensional global attractor for the
damped and driven sine-Gordon equation, {\it Sci.China} (Ser A) {\bf
41(2)} (1998), 113-122.

\vspace{-0.08in}\bibitem{QZQ} M. Qian, S. Zhu and W.X. Qin, Dynamics in a chain of overdamped pendula driven by constant
 torques, {\it SIAM  J. Appl. Math.} {\bf 57} (1997), 294-305.

\vspace{-0.08in}\bibitem{S} W. Shen, Global attractor and rotation number of a class of
nonlinear noisy oscillators, {\it Disc. Cont. Dyn. Syst.} {\bf 18} (2007), 597-611.

\vspace{-0.08in}\bibitem{SZS} Z. Shen, S. Zhou and W. Shen, One-Dimensional Random Attractor and Rotation Number
of the Stochastic Damped Sine-Gordon Equation,  {\it J. Differential Equations}
{\bf  248}  (2010),   1432-1457.



\vspace{-0.08in}\bibitem{Tem} R. Temam, Infinite-Dimensional Dynamical Systems in Mechanics and Physics, Applied Mathematical
Sciences {\bf 68},
Springer-Verlag, New York (1988).


\vspace{-0.08in}\bibitem{WaZh1} G. Wang and S. Zhu, On the Dimension of the global attractor for the damped sine-Gordon equation,
{\it J. Math. Phy.} {\bf 38(6)} (1997), 3137-3141.

\vspace{-0.08in}\bibitem{WiHa} K. Wiesenfeld and P. Hadley, Attractor crowding in
oscillator arrays, {\it Phys. Rev. Lett.} {\bf 62} (1988),
1335-1338.

\vspace{-0.08in}\bibitem{Z} S. Zhou, Attractor and dimension for discretization of a
damped wave equation  with periodic nonlinearity,  {\it Topologic
Methods in Nonlinear Analysis}, {\bf 15} (2000),  267-281.

\vspace{-0.08in}\bibitem{ZhYiOu} S. Zhou, F. Yin, and Z. Ouyang,
 Random attractor for damped nonlinear wave equations with white noise, {\it
  SIAM J. Appl. Dyn. Syst.} {\bf  4} (2005),  883-903.

\end{thebibliography}
\end{document}